\documentclass[10pt]{article}
 \usepackage[latin1]{inputenc}
 \usepackage[dvips]{graphicx}
 \usepackage{wrapfig}
 \usepackage{amsmath}
 \usepackage{amsthm}
 \usepackage{amsfonts}
 \usepackage{amssymb}
 \usepackage{amsopn}
 \usepackage{layout}
 \usepackage{verbatim}
 \usepackage{alltt}
 \usepackage{XYpic}
 \input xy
 \xyoption{all}
 
\oddsidemargin - 1.5 pt
\newcounter{thm}[section]

\renewcommand{\thethm}{\arabic{section}.\arabic{thm}}

\newenvironment{con}{\par\medskip\noindent\refstepcounter{thm}
\bgroup{\hspace*{-0.15 cm}\bf{Conjecture}
\thethm.}\bgroup\it}{\egroup \egroup\par\medskip}

\newenvironment{thm}{\par\medskip\noindent\refstepcounter{thm}
\bgroup{\hspace*{-0.15 cm}\bf{Theorem}
\thethm.}\bgroup\it}{\egroup \egroup\par\medskip}

\newenvironment{lemma}{\par\medskip\noindent\refstepcounter{thm}
\bgroup{\hspace*{-0.15 cm}\bf{Lemma} \thethm.}\bgroup\it}{\egroup
\egroup\par\medskip}

\newenvironment{prop}{\par\medskip\noindent\refstepcounter{thm}
\bgroup{\hspace*{-0.15 cm}\bf{Proposition}
\thethm.}\bgroup\it}{\egroup \egroup\par\medskip}

\newenvironment{cor}{\par\medskip\noindent\refstepcounter{thm}
\bgroup{\hspace*{-0.15 cm}\bf{Corollary}
\thethm.}\bgroup\it}{\egroup \egroup\par\medskip}

\newenvironment{dfn}{\par\medskip\noindent\refstepcounter{thm}
\bgroup{\hspace*{-0.15 cm}\bf{Definition}
\thethm.}\bgroup}{\egroup \egroup\par\medskip}

\begin{document}

  \author{Seyfi Türkelli}

  \title{Connected Components of Hurwitz Schemes and Malle's Conjecture}
  \maketitle

\footnotetext[1]{{\it 2000 Mathematics Subject Classification}. Primary ; Secondary }
\footnotetext[2]{{\it Key words and phrases.} Hurwitz Schemes, Branch covers, Malle's Conjecture }

\begin{abstract}
Let $\mathcal{Z}(X)$ be the number of degree-d extensions of $\mathbb{F}_q(t)$ with bounded discriminant and some
specified Galois group.  The problem of computing $\mathcal{Z}(X)$ can be related to a problem of
counting $\mathbb{F}_q$-rational points on certain Hurwitz spaces.  Ellenberg and Venkatesh used
this idea to develop a heuristic for the asymptotic behavior of $\mathcal{Z}'(X)$, the number of
-geometrically connected- extensions, and showed that this agrees with the conjectures of
Malle for function fields.  We extend Ellenberg-Venkatesh's argument to handle the more
complicated case of covers of $\mathbb{P}^1$ which may not be geometrically connected, and show that
the resulting heuristic suggests a natural modification to Malle's conjecture which
avoids the counterexamples, due to Kl\"uners, to the original conjecture.
\end{abstract}

\section{Introduction}

 Let $k$ be a  number field and let $N\leq S_n$ be a transitive group with one-point stabilizer $H\leq N$. By an $N$\emph{-extension}, we mean a Galois extension $K/k$ with Galois group $G(K/k)\cong N$. We will denote the discriminant of a finite extension $K/k$ by $D(K/k)$. It is well known that the number of extensions $K/k$ with $\textbf{N}^k_{\mathbb{Q}}D(K/k)<X$ is finite. In \cite{ma2}, Malle conjectures an asymptotic formula for the number of $N-$extensions $K/k$ of a fixed number field $k$ with $\textbf{N}^k_{\mathbb{Q}}D(K^H/k)<X$ where $K^H/k$ is the intermediate extension corresponding to $H$. In order to state Malle's conjecture precisely, we need to introduce some invariants of the group $N$.

 Let $g\in N$. We define the \emph{index} of $g$ to be the number $ind(g)=n-r$ where $r$ is the number of orbits of $g$ on set $\{1,2,...,n\}$. We define the \emph{index} of the group $N$, $ind(N)$, to be the minimum of all $\{ind(g)\mid g\in N^{\#} \}$ where $N^{\#}=N-\{1\}$. Finally, we define our first invariant $a(N)=1/ind(N)$.

 In order to define our second invariant, we let $\mathcal{C}(N)$ be the set of conjugacy classes of $N^{\#}$ whose index is equal to the index of $N$ (the ones with the minimal index). We define a $G(\bar{k}/k)$ action on set $\mathcal{C}(N)$ via the cyclotomic character $\chi$ as $g.c:=c^{\chi(g)}$ for $g\in N$ and $c\in \mathcal{C}(N)$. Now, we define our second invariant to be the positive integer $b(N,k)=|\mathcal{C}(N)/G(\overline{k}/k)|$.

 Fix $k$ and denote the number of $N$-extensions $K/k$ with $\textbf{N}^k_{\mathbb{Q}}D(K^H/k)<X$ by $\mathcal{Z}_N(k,X)$. \emph{Malle's conjecture} \cite{ma2} is stated as follows:

\begin{con}(Malle)\label{malle_con}
Let $k$ be a number field and $N$ be a transitive subgroup of $S_n$. Then,
$$\mathcal{Z}_N(k,X) \asymp {X^{a(N)}}(logX)^{b(N,k)-1}.$$
\end{con}

 Note that this conjecture is known for abelian groups and $k=\mathbb{Q}$ by the work of Wright \cite{wr}, for $N=S_3$ by the work of Davenport-Heilbronn \cite{da}, for $N=D_4,S_4,S_5$ and $k=\mathbb{Q}$ by the work of Bhargava, see \cite{bh1} and \cite{bh2}. In \cite{kl1}, Kl\"uners and Malle proved that for each positive $\epsilon >0$ there are positive constants $c_{\epsilon},C_{\epsilon}$ such that $c_{\epsilon} X^{a(N)} < \mathcal{Z}_N(\mathbb{Q},X) < C_{\epsilon} X^{a(N)+\epsilon}$ for any Nilpotent group $N$ given with its regular representation $N \hookrightarrow S_{|N|}$.

 Recently, Kl\"uners gave a counterexample to this conjecture, see \cite{kl2}. Indeed, take $k=\mathbb{Q}$ and $N=(\langle(123)\rangle\oplus \langle(456)\rangle)\rtimes \langle(14)(25)(36)\rangle\leq S_6$
which is isomorphic to $C_3\wr C_2$ where $C_r$ is the cyclic group with $r$ elements. Then, one can easily see that the conjecture predicts that $\mathcal{Z}_N(\mathbb{Q},X) \asymp X^{1/2}$. Kl\"uners shows that $\mathcal{Z}_N(\mathbb{Q},X) \asymp X^{1/2}logX$. He also points out that if one counts only the {\it regular} extensions {\it i.e.} extensions without an intermediate cyclotomic field, then one gets the asymptotic in the conjecture.

 For the rest of the paper, fix a $q$ which is coprime to $|N|$. One can state the conjecture for $k=\mathbb{F}_q (t)$ with evident modifications. In this setting, constant intermediate fields correspond to intermediate cyclotomic extensions. Subject to some heuristics on Hurwitz schemes, Ellenberg and Venkatesh \cite{el2} compute the size of the main term in the asymptotic for the number of $N$-extensions without constant subextensions, and obtain the analogue of Malle's conjecture.

 This suggests that Malle's conjecture may correctly predict the asymptotics for extensions without cyclotomic (constant) subextensions. Using the idea of Ellenberg-Venkatesh \cite{el2}, we will count $N$-extensions $K/\mathbb{F}_q (t)$ with a maximal constant subextension corresponding to a fixed normal subgroup $G$; we will call such an extension $N_G$\emph{-extension}. Then, we will take the maximum of the asymptotics of $N_G$-extensions as $G$ varies to get the asymptotic for $\mathcal{Z}_N(\mathbb{F}_q(t),X)$. We will propose a conjecture eliminating the Kl\"uners' counterexamples and compatible with the existing results.

 The idea is as follows: the category of $N_G$-extensions of $\mathbb{F}_q (t)$ is equivalent to the category of (connected but not necessarily geometrically connected) $N_G$-covers of $\mathbb{P}^1_{\mathbb{F}_q}$. Therefore, counting extensions is more or less equivalent to counting $\mathbb{F}_q$-rational points on the moduli space of certain covers of $\mathbb{P}^1$, namely Hurwitz schemes. Assuming the {\it heuristic}:
 \ \\

\label{heu}(A) \space Let $\mathcal{H}$ be a geometrically connected scheme of dimension $d$ defined over $\mathbb{F}_q$; then one has $|\mathcal{H}(\mathbb{F}_q)|=q^d$,
\ \\

\noindent we reduce the problem to one of counting irreducible components of Hurwitz spaces, and computing their dimension.

 More precisely, let $\mathcal{Z}_{G,N}(\mathbb{F}_q,X)$ be the number of $N_G$-extensions $K/\mathbb{F}_q(t)$ with $Norm D(K^H/\mathbb{F}_q(t))<X$. We will prove:

 \begin{thm}\label{main_thm}
 Assume (A). If $G$ splits in $N$, then $$\mathcal{Z}_{G,N}(\mathbb{F}_q,X)\asymp X^{a(G)}(logX)^{b(G,N,\mathbb{F}_q)-1} $$ where $b(G,N,\mathbb{F}_q)$ is an explicitly computable positive integer.
 \end{thm}

 As an immediate consequence, we have the result of Ellenberg-Venkatesh:

 \begin{cor}\label{ell_ven}
 Assume (A). Then, $\mathcal{Z}_{N,N}(\mathbb{F}_q,X)\asymp X^{a(N)}(logX)^{b(N,\mathbb{F}_q)-1}.$
 \end{cor}

 In section 4, we will conjecture an asymptotic for the number of $N$-extensions of a global field. As evidence in the favor of our modification of Malle's conjecture, we show that our version is not contradicted by Kl\"uners' counterexample from \cite{kl2}.

 \begin{cor}\label{counterexample}
 Assume (A). Let $N=(\langle(123)\rangle\oplus \langle(456)\rangle)\rtimes \langle(14)(25)(36)\rangle\leq S_6$ and $q=2$ $(mod3)$. Then, we have $\mathcal{Z}_{N}(\mathbb{F}_q,X)\asymp X^{1/2}logX.$
 \end{cor}

 We want to count branch covers of $\mathbb{P}^1_{\mathbb{F}_q}$ corresponding to $N_G$-extensions of $\mathbb{F}_q(t)$, namely $N_G$\emph{-covers}. These covers are parameterized by certain $\mathbb{F}_q$-rational points of Hurwitz schemes of $(G,N)$-covers of $\mathbb{P}^1$. In section 2, we will state well known facts about these moduli schemes. In particular, we will introduce discrete invariants, so called Nielsen tuples, parameterizing "almost" geometrically connected components of Hurwitz schemes and we will determine the components parameterizing $N_G$-covers. In section 3, using the heuristic, we will reduce the problem of counting covers to counting Nielsen tuples. In section 4, we prove Theorem \ref{main_thm} and we conjecture an asymptotic for $\mathcal{Z}_N(k,X)$ where $k$ is a global field. Finally, we give some examples, as corollaries of Theorem \ref{main_thm}, supporting our conjecture.
\ \\

 {\it Acknowledgments}. \indent The author is very grateful to Jordan Ellenberg for his excellent guidance in preparation of this paper, and thanks to J\"urgen Kl\"uners for his very useful comments and pointing out an important mistake in an earlier version of the paper.

%-----------------------------------------------------------------------------------------------------------------------

\section{Hurwitz Schemes of $(G,N)$-covers}

 For the rest of the paper, fix a transitive subgroup $N\leq S_n$ and a normal subgroup $G$ of $N$ of cardinality $m$ with cyclic quotient $N/G$. Also, fix a one-point stabilizer $H\leq N$. Let $q$ be such that $(q,|N|)=1$.
 
 In this paper, we will prove the asymptotic formula under the assumption that $G$ splits in $N$, that is $N=G\rtimes T$ for some cyclic group $T$. We hope to treat the general case in a future paper. So, once and for all, fix a cyclic complement $T$ of $G$ in $N$ and an element $\tau\in T$ generating $T$. Let $C_T$ be the kernel of the projection $T\tilde{\rightarrow} N/G\twoheadrightarrow N/GCen_N(G)$, $T':=T/C_T\cong N/GCen_N(G)$ and $\tau_1\in T'$ be the image of $\tau$. Set $|T|=d$ and $|T'|=d'$.  

 In this section, first, we will introduce $(G,N)$-\emph{covers} of $\mathbb{P}^1/\mathbb{F}_q$ and discuss basic facts about their moduli spaces, namely Hurwitz schemes. Secondly, we will discuss the decomposition of Hurwitz schemes into their almost-geometrically-connected components. Finally, we will determine the components which parameterize the covers we want to count, namely $N_G$-\emph{covers}. Most of the work we summarize in this section is due to Fried \cite{fr} and Wewers \cite{we}.

\begin{dfn}\label{defn_gn_extn}
 An $N_G$-\emph{extension} of $\mathbb{F}_q(t)$ is an $N$-extension $K/\mathbb{F}_q(t)$ such that $K^G/\mathbb{F}_q(t)$ is the maximal constant subextension of $K/\mathbb{F}_q(t)$ where $K^G$ is the subfield of $K$ corresponding to group $G$.
\end{dfn}

 Let $\rho: N\rightarrow S_m$ be the representation given by the action of $N$ on the left coset space $N/T$. Note that, for $g\in G$, $ind(g)$ will denote the index of $g\in S_n$, not the index of $\rho(g)$ in $S_m$.

 Let $Y/\mathbb{F}_q$ be a geometrically connected smooth curve and $f:Y\rightarrow \mathbb{P}^1/\mathbb{F}_q$ be a finite branched cover of degree $m=|G|$. Note that the field of moduli of the cover $(Y,f)$ is also the field of definition, see \cite[Corollary 3.3]{de1}. Let $K/\mathbb{F}_q$ be the field of definition of the cover $(Y,f)$ and let $f_K:Y_K\rightarrow \mathbb{P}^1_K$ be one of its models over $K$ with branch locus $B$ where $B$ is a proper closed reduced subscheme of $\mathbb{P}^1_K$ of degree $k$. Consider pairs $(b,\sigma)$ where $b$ is a geometric point on $\mathbb{P}^1/\mathbb{F}_q$  and $\sigma:\{1,...,m\} \rightarrow f^{-1}(b)$ is a bijection such that
 \begin{enumerate}
 \item[i.] The image of the homomorphism $\tilde{\phi}:\Pi_1(\mathbb{P}^1_K\backslash B,b)\rightarrow S_m$ induced by $f$ and $\sigma$ is included in $\rho(N)$;
 \item[ii.] The image of $\phi:\Pi_1(\mathbb{P}^1_{\bar{\mathbb{F}}_{q}}\backslash B,b)\rightarrow S_m$ is $\rho(G)$
 \end{enumerate}
where $f^{-1}(b)$ denotes the geometric fiber. Two pairs $(b,\sigma)$ and $(b',\sigma ')$ are $N$-\emph{equivalent} if there is a path $\gamma\in \Pi_1(\mathbb{P}^1_K\backslash B;b,b')$ such that $\sigma^{-1}\tilde{\gamma}\sigma '\in \rho(N)$ where $\tilde{\gamma}:f^{-1}(b')\rightarrow f^{-1}(b)$ is the bijection obtained from lifting $\gamma$ to $Y$.

\begin{dfn}
 A \emph{$(G,N)$-structure on $f$} is the $N$-equivalence class $\mathcal{S}=[(b,\sigma)]_N$ of a pair $(b,\sigma)$.   A $(G,N)-cover$ is a pair $(f,\mathcal{S})$ where $f:Y\rightarrow \mathbb{P}^1/\mathbb{F}_q$ is a cover as above and $\mathcal{S}$ is a $(G,N)-$structure on $f$.
\end{dfn}

   We remark that $f:Y\rightarrow\mathbb{P}^1/\mathbb{F}_q$ is not a Galois cover in general but $\bar{f}=f\otimes \bar{\mathbb{F}}_q:\bar{Y}\rightarrow \mathbb{P}^1_{\bar{\mathbb{F}}_q}$ is Galois, see \cite[Remark 3.4.4]{we}.

 Abusing the notation, by $Im(\tilde{\phi})$ ( or $Im(\phi)$ ) we mean homomorphic image of $\tilde{\phi}$ ( or $Im(\phi)$ ) in $S_n$. In other words, a $(G,N)$-cover is a cover of $\mathbb{P}^1/\mathbb{F}_q$ with \emph{geometric fundamental group}, $Im(\phi)$, $G$ and with \emph{arithmetic fundamental group}, $Im(\tilde{\phi})$, at most $N$. Note that we want to count covers with arithmetic fundamental group exactly $N$.

\begin{dfn}
 An $N_G-cover$ is a $(G,N)$-cover $(f,\mathcal{S})$ with $Im(\tilde{\phi})=N$.
\end{dfn}

 Note that given a geometric point $b$ on $\mathbb{P}^1/{\mathbb{F}_q}$ one can represent any $(G,N)$-structure $[(b',\sigma')]_N$ by $[(b,\sigma)]_N$ for some $\sigma$ because $\mathbb{P}^1/\mathbb{F}_q$ is path connected.

\begin{dfn}
 A morphism between two $(G,N)$-covers $(f,[(b,\sigma)]_N)$ and $(f',[(b,\sigma')]_N)$ is a morphism $\psi:Y\rightarrow Y'$ of schemes over $\mathbb{F}_q$ such that $f=f'\psi$ and $[(b,\sigma)]_N=[(b,\psi^{-1}\sigma')]_N$.
\end{dfn}

 \begin{thm}\label{thm_modspace}
There exists a smooth scheme $\mathcal{H}_{G,N}$ over $\mathbb{Z}[1/|G|]$ which is a coarse moduli scheme for $(G,N)$-covers of $\mathbb{P}^1$. Moreover, the fibers of the natural map
$$\xymatrix{\Delta : \{ (G,N)\text{-covers of } \mathbb{P}^1 \text{ defined over } \mathbb{F}_q \}/ \cong \ar[r] &  \mathcal{H}_{G,N}(\mathbb{F}_q)} $$
has size at most $|Cen_N(G)|$.
\end{thm}

\begin{proof}
 Wewers carries out a detailed construction of $\mathcal{H}_{G,N}$ in \cite[Section 4]{we}. Let $P\in \mathcal{H}_{G,N}(\mathbb{F}_q)$. By section $4.1.3$ of \cite{we} and \cite{de1}, the obstruction to $P$ arising from a cover lies in the second cohomology group $H^2(\mathbb{F}_q,Cen_N(G))$. Since $\mathbb{F}_q$ has cohomological dimension one, there is no obstruction. By Lemma \ref{lemma_par} below, $\Delta^{-1}(P)$ is parameterized by $H^1(\mathbb{F}_q, Cen_N(G))$. This cohomology group has size at most $|Cen_N(G)|$.
\end{proof}

\begin{lemma}\label{lemma_par}
The isomorphism classes of $(G,N)$-covers of $\mathbb{P}^1$
parameterized by the point $P\in \mathcal{H}_{G,N}(\mathbb{F}_q)$
is in one-to-one correspondence with the cohomology group
$H^1(\mathbb{F}_q,Cen_N(G))$.
\end{lemma}

\begin{proof}
 Let $P\in \mathcal{H}_{G,N}(\mathbb{F}_q)$ and let $(f,[b,\sigma]_N)$ be a $(G,N)$-cover defined over $\mathbb{F}_q$ corresponding to the point $P$. Let $(f_1,[b,\sigma_1])$ be another one. Then, there is an isomorphism of $(G,N)$-covers $\psi:\bar{Y}\rightarrow \bar{Y}_1$ with $\bar{f}=\bar{f}_1\psi$ and $[b,\sigma]_N=[b,\psi^{-1}\sigma_1]_N$ over $\bar{\mathbb{F}}_q$ where $\bar{f}_i=f_i\times _{\mathbb{F}_q}\bar{\mathbb{F}}_q$ for $i=1,2$.
\noindent This induces a map
$$\delta _{\psi}: G_{\mathbb{F}_q}\rightarrow N
\text{\indent  defined by \indent } x\mapsto (\psi^{-1}\psi^x)^\sigma .$$

 It is easy to see that $\psi^{-1}\psi^x\in Aut((\bar{f},[b,\sigma]_N))$ for all $x\in G_{\mathbb{F}_q}$. Since $Aut((\bar{f},[b,\sigma]_N))\cong Cen_N(G)$,  Im($\delta_\psi$)$\subseteq Cen_N(G)$ \cite[Lemma 3.4.2]{we}. One checks that $\delta _{\psi}$ is a cocycle and $\delta$ defines the desired correspondence.

\end{proof}

 Later, we will make use of the important fact that $\mathcal{H}_{G,G}\cong \mathcal{H}_G$ as schemes defined over $\mathbb{F}_q$ where $\mathcal{H}_G$ is the coarse moduli scheme of geometrically connected Galois $G$-covers of $\mathbb{P}^1$, see \cite[proposition 4.2.2]{we}.

 Let $Ni_k(G)$ be the set of $k$-tuples $\bar{g}=(g_1,...,g_k)$ generating $G$ and satisfying $g_1...g_k=1$. Note that $N$ acts on the set of such $k$-tuples by component-wise-conjugation; denote the orbit of $\bar{g}$ in this action by $[\bar{g}]_N$. Denote the set of such $N$-orbits of $k$-tuples by $Ni_k(G,N)$ and define the set of $G-N-$\emph{Nielsen classes} to be the set $Ni(G,N):=\bigcup_k Ni_k(G,N)$.

 Given a $G-N-$Nielsen class $[\bar{g}]_N\in Ni_k(G,N)$ and a proper closed reduced subscheme $B$ of $\mathbb{P}^1_{\bar{\mathbb{F}}_{q}}$ of degree $k$, one has surjective homomorphisms $\pi_1(\mathbb{P}_{\bar{\mathbb{F}}_q}^1\backslash B,b_0)\rightarrow G$ modulo conjugation with elements of $N$. After fixing a set of generators of $\pi_1(\mathbb{P}_{\bar{\mathbb{F}}_q}^1\backslash B,b_0)$, one gets a one-to-one correspondence between these homomorphisms and $(G,N)$-covers$/\bar{\mathbb{F}}_q$ with branch locus $B$. Thus, such a orbit $[\bar{g}]_N\in Ni(G,N)$ and branch locus $B$ induces a $(G,N)$-cover and vice versa.

 Let $\mathcal{H}_{G,N,k}$ be the moduli space of $(G,N)$-covers with degree-$k$ branch locus and $U_k=\mathbb{P}^k\backslash \delta_k$ where $\delta_k$ is the discriminant locus. Then, the natural map $\pi: \mathcal{H}_{G,N,k}\rightarrow U_k$ taking a cover to its branch locus is \'etale. On the other hand, if $B$  is a geometric point of $U_k/\mathbb{F}_q$, then the braid group $\pi_1(\bar{U}_k,B)$, the geometric fundamental group of $U_k$, can be written in standard representation with generators $Q_1,...,Q_{k-1}$.

 In summary, there is a bijection between $Ni_k(G,N)$ and the geometric fiber $\pi^{-1}(B)$ \cite[Proposition 4.3.1]{we}. This bijection induces a well-known action of the braid group $\pi_1(\bar{U}_k,B)$ on $Ni_k(G,N)$ which is given by $$Q_i[(g_1,...,g_k)]_N:=[(g_1,...,g_i g_{i+1} g_i^{-1},g_i,...g_k)]_N$$ and the geometrically connected components of $\mathcal{H}_{G,N,k}$ correspond to the braid group orbits on $Ni_k(G,N)$.

 Given a $k$-tuple $\bar{C}=(C_1,...,C_k)$ of $G$-conjugacy classes, one can consider its orbit $[\bar{C}]_N$ under conjugation by $N$. Define $Ni(\bar{C})$ to be the set of $\bar{g}\in Ni_k(G)$ such that, after some permutation of the entries of $\bar{g}$, $g_i\in C_i$ for all $i$. Now, define $Ni([\bar{C}]_N)$ to be the set of $N-$orbits of elements of $Ni(\bar{C})$. Clearly, $Ni([\bar{C}]_N)$ is closed under the braid group action. Therefore, there is a closed subscheme (possibly empty) of $\mathcal{H}_{G,N}$ corresponding to the $N$-orbit $[\bar{C}]_N$, denoted by $\mathcal{H}_{[\bar{C}]_N}$, which parameterizes $(G,N)$-covers with ramification data $[\bar{C}]_N$. Thus, we have the decomposition $\mathcal{H}_{(G,N)}=\bigsqcup_{[\bar{C}]_N}\mathcal{H}_{[\bar{C}]_N}$ over $\bar{\mathbb{F}}_q$. For  more details of the discussion above, we refer to \cite[Section 4.3]{we}.

 In summary, we have:

 \begin{thm}\label{thm_decomp}
 \cite[Section 4]{we} For each $N$-orbit $[\bar{C}]_N$ of a tuple of conjugacy classes $\bar{C}$, there is a Hurwitz scheme $\mathcal{H}_{[\bar{C}]_N}/\bar{\mathbb{F}}_q$ which is a coarse moduli scheme for $(G,N)$-covers $Y\rightarrow \mathbb{P}^1_{\bar{\mathbb{F}}_q}$ with ramification data $[\bar{C}]_N$. A Galois transformation $\sigma\in G_{\mathbb{F}_q}$ acts on $\mathcal{H}_{[\bar{C}]_N}$ by  $\mathcal{H}^\sigma_{[\bar{C}]_N}=\mathcal{H}_{[\bar{C}]_N^\sigma}$.

 Moreover, the map $\pi: \mathcal{H}_{[\bar{C}]_N}\rightarrow U_k$ sending a cover to its branch locus is \'etale and geometric points of the fiber over $B\in U_k(\bar{\mathbb{F}}_q)$ can be identified with $Ni([\bar{C}]_N)$. The induced action of the monodromy group $\Pi_1(U_k)$ on $Ni([\bar{C}]_N)$ is given by
$$Q_i(g_1,...,g_k)=(g_1,..,g_i g_{i+1} g_i^{-1},g_i,..,g_k)$$
and, thus geometrically connected components of $\mathcal{H}_{[\bar{C}]_N}$ correspond to $\Pi_1(U_k)$-orbits of $Ni([\bar{C}]_N)$.
\end{thm}

 Note that the arithmetic fundamental group of a $(G,N)$-cover is a subgroup of $N$ containing $G$ and we want to count the ones defined over $\mathbb{F}_q$ whose arithmetic fundamental group is exactly $N$, namely $N_G$ covers defined over $\mathbb{F}_q$. Note also that an $N_G$-cover defined over $\mathbb{F}_q$ corresponds to a $N_G$-extension of $\mathbb{F}_q(t)$ and vice versa.

 Let $\mathcal{H}_{G,N}(\mathbb{F}_q)^N$ be the subset of $\mathcal{H}_{G,N}(\mathbb{F}_q)$ consisting of all the points parameterizing $N_G$-covers defined over $\mathbb{F}_q$; that is
$$\mathcal{H}_{G,N}(\mathbb{F}_q)^N:=\Delta(\{N_G\text{-covers of } \mathbb{P}^1 \text{ defined over }\mathbb{F}_q \}/\cong)$$
\noindent where $\Delta$ is the natural map in Theorem
\ref{thm_modspace}. Below, we determine the components whose
$\mathbb{F}_q$-rational points parameterize $N_G$-covers defined
over $\mathbb{F}_q$.

  For any given subgroup $N'\subseteq N$ containing $G$, the natural map
$$\Lambda_N^{N'}:\mathcal{H}_{G,N'}\rightarrow \mathcal{H}_{G,N} \text{\indent induced by \indent} (f,[(b,\sigma)]_{N'})\mapsto (f,[(b,\sigma)]_{N})$$
\noindent is an \'etale cover of degree $|Out_N(N')|=| N/N'Cen_N(N')|$ and the map $\pi': \mathcal{H}_{G,N'}\rightarrow U_k$ factors through $\pi: \mathcal{H}_{G,N}\rightarrow U_k$, see \cite[Section 6.1]{fr} and \cite[Section 4.2]{we}. Note that the image of $\mathcal{H}_{[\bar{C}]_{N'}}$ under $\pi$ is $\mathcal{H}_{[\bar{C}]_N}$. In fact, if $\mathcal{H}_{[\bar{C}]_{N'}}$ is connected and defined over $\mathbb{F}_q$, then $\Lambda_N^{N'}$ is a Galois cover with automorphism group isomorphic to $Out_N(N')$ where $Out_N(N')$ denotes the image of $N$ in $Out(N')$.

 In particular, by taking $N'=G$ we obtain an \'etale cover $\Lambda:\mathcal{H}_{G}\rightarrow \mathcal{H}_{G,N}/\mathbb{F}_q$ of degree $d'=|T'|=|T/C_T|$ (recall that there is a canonical isomorphism $\mathcal{H}_{G}\cong\mathcal{H}_{G,G}/\mathbb{F}_q$). Given a geometric point $P$ in $\mathcal{H}_{G,N}/\mathbb{F}_q$, $C_T$ acts on the geometric fiber $\Lambda^{-1}(P)$ trivially and this induces a sharply transitive action of $T'$ on $\Lambda^{-1}(P)$.

 Let $P\in \mathcal{H}_{G,N}(\mathbb{F}_q)$ and $Q\in \Lambda^{-1}(P)$. Then, the geometric fiber $\Lambda^{-1}(P)$ is defined over $\mathbb{F}_q$ and for all $\sigma\in G_{\mathbb{F}_q}$ there exists a unique $\zeta(\sigma)\in T'$ such that $Q^{\sigma}=Q^{\zeta(\sigma)}\in \Lambda^{-1}(P)$. One can easily see that $\zeta\in H^1(\mathbb{F}_q,T')$ (here the action of the absolute Galois group $G_{\mathbb{F}_q}$ on $T'$ is trivial so indeed $\zeta\in Hom(G_{\mathbb{F}_q},T')$) and it is independent of the choice of $Q\in \Lambda^{-1}(P)$. Therefore, $\Lambda^{-1}(P)\subseteq \mathcal{H}^{\zeta}_G(\mathbb{F}_q)$ where $\mathcal{H}^{\zeta}_G$ is the twist of $\mathcal{H}_G$ via the composition of $\zeta$ with the embedding of $T'$ in $Aut(\mathcal{H}_G/ \mathcal{H}_{G,N})$. On the other hand, given a point $Q\in \mathcal{H}^{\zeta}_G(\mathbb{F}_q)$ for some $\zeta\in H^1(\mathbb{F}_q,T')$,  $P:=\Lambda(Q)\in \mathcal{H}_{G,N}(\mathbb{F}_q)$. Thus, we have $$\Lambda^{-1}(\mathcal{H}_{G,N}(\mathbb{F}_q))=\bigsqcup_{\zeta\in H^1(\mathbb{F}_q,T')}\mathcal{H}^{\zeta}_G(\mathbb{F}_q).$$

 Now, we want to determine the cocycles $\zeta\in H^1(\mathbb{F}_q,T')$ for which the twists $\mathcal{H}^{\zeta}_G(\mathbb{F}_q)$ induce $N_G$-covers via $\Lambda$. So, let $|T'|=d'$ and let $e$ be a positive integer with $1\leq e\leq d'$. We denote the $1$-cocycle sending $Frob_q$ to $\tau_1^e$ by $\zeta_e\in H^1(\mathbb{F}_q,T')$ where $\tau_1$ is the image of $\tau\in T$ under the projection $T\twoheadrightarrow T'$. The following proposition tells us the schemes we should consider.

\begin{prop}\label{prop_dec}
We have the following decomposition
$$\Lambda^{-1}(\mathcal{H}_{G,N}(\mathbb{F}_q)^N)= \bigsqcup_{\substack{1\leq e \leq d' \\  (e,d')=1}}\mathcal{H}^{\zeta_e}_G(\mathbb{F}_q).$$
\noindent
\end{prop}

\begin{proof}
 Giving a $(G,N)$-cover defined over $\mathbb{F}_q$ with branch locus $B$ is equivalent to giving the following diagram of fundamental groups:

 $$\xymatrix{ 1\ar[r] & \Pi_1(\mathbb{P}^1_{\bar{\mathbb{F}}_{q}}\backslash
B,b)\ar[r]\ar@{>>}[d]_{\phi} &
\Pi_1(\mathbb{P}^1_{\mathbb{F}_{q}}\backslash
B,b)\ar[r]\ar[d]_{\tilde{\phi}} &
G(\bar{\mathbb{F}}_{q}/\mathbb{F}_{q})\ar[r]\ar[d]_{\bar{\phi}} & 1
\\ 1 \ar[r] & G \ar[r] & N \ar[r] & T }.$$

\noindent The cover corresponding to the above diagram is actually an $N_G$-cover if and only if $\bar{\phi}$ is surjective \emph{i.e.} $\bar{\phi}(Frob_q)=\tau^e$ for some positive integer $e$ with $1\leq e \leq d$ and $(e,d)=1$.

 Let $P\in \mathcal{H}_{G,N}(\mathbb{F}_q)^N$  and let $Q\in \Lambda^{-1}(P)$. Then, by the definition of $\Lambda$, $$Q^{Frob_q}=Q^{\tau_1^e}$$ for some unique integer $e$ with $$1\leq e \leq d' \text{ \indent and\indent } (e,d')=1.$$ Therefore, $Q\in \mathcal{H}^{\zeta_e}_G(\mathbb{F}_q)$ and  $\Lambda^{-1}(\mathcal{H}_{G,N}(\mathbb{F}_q)^N)\subseteq \bigsqcup_{\substack{1\leq e \leq d' \\  (e,d')=1}}\mathcal{H}^{\zeta_e}_G(\mathbb{F}_q).$ The difficult part is to prove the inequality "$\supseteq$"; that's what we do below.

  Let $e$ be a positive integer with $(e,d')=1$ and $1\leq e\leq d'$ and let $Q\in \mathcal{H}^{\zeta_e}_G(\mathbb{F}_q)$. We want to show that $Q\in \Lambda^{-1}(\mathcal{H}_{G,N}(\mathbb{F}_q)^N).$

 %Assume $Cen_N(G)\subseteq G$. Then $P:=\Lambda(Q)$ corresponds a diagram as above such that $\phi^{Frob_q}=\phi^{\tau^eg}$ for some $g\in G$. Since $Cen_N(G)\subseteq G$, $\bar{\phi}(Frob_q)=\tau^e$. As $(e,d)=1$, $\bar{\phi}$ is surjective. Therefore, $P$ parameterizes an $N$-cover \emph{i.e.} $Q\in \Lambda^{-1}(\mathcal{H}^N_{G,N}(\mathbb{F}_q)).$

 In other words, we want to show that there exists an $N_G$-cover $(f,\mathcal{S})$ such that $\Delta((f,\mathcal{S}))=P$ where  $P=\Lambda(Q)$ and $$\xymatrix{\Delta : \{ (G,N)\text{-covers of } \mathbb{P}^1 /\mathbb{F}_q \}/ \cong \ar[r] &  \mathcal{H}_{G,N}(\mathbb{F}_q)} $$ is the natural map defined in Theorem \ref{thm_modspace}.

  Let $(f_0,\mathcal{S}_0)\in \Delta^{-1}(P)$. Then $(f_0,\mathcal{S}_0)$ is an $M_G$-cover for some subgroup $M\subseteq N$ containing $G$. Let $T_0\subseteq T$ be the complement of $G$ in $M$, that is $M=G\rtimes T_0$. $(f_0,\mathcal{S}_0)$ induces a diagram as above, which is induced by a surjective morphism of groups $\tilde{\phi}_0: \Pi_1(\mathbb{P}^1_{\mathbb{F}_{q}}\backslash B,b)\twoheadrightarrow M$. Then, $\bar{\phi}_0(Frob_q)=\tau^{e'}$ for some $e'$ with $1\leq e'\leq d$. This implies that $Q^{Frob_q}=Q^{\tau^{e'}}$. On the other hand, $Q^{Frob_q}=Q^{\tau^e}$ since $Q\in \mathcal{H}^{\zeta_e}_G(\mathbb{F}_q)$. Therefore, $Q^{\tau^{e-e'}}=Q$ and $\tau^{e-e'}\in C_T$, that is $\tau^e=\tau^{e'}$ in $T/C_T$. Since $(e,d')=1$, $\tau^e$ generates $T/C_T\cong N/GCen_N(G)$ and so does $\tau^{e'}$. Thus, we have natural isomorphisms:
$$N/Cen_N(G)\cong M/Cen_{M}(G) \indent \text{and} \indent N/GCen_N{G}\cong M/GCen_{M}(G).$$

 Recall that $|T/C_T|=d'$ and $|C_T|=d''$. One can easily see that $(e+ad',d)=1$ for some $a$ with $0\leq a \leq d''-1$. So, let $e''=e+ad'<d$ be such that $(e'',d)=1$. Note that $\tau^e=\tau^{e'}=\tau^{e''}$ in $N/GCen_N(G)$.

 Let $\phi:=\phi': \Pi_1(\mathbb{P}^1_{\bar{\mathbb{F}}_{q}}\backslash B,b)\twoheadrightarrow G$ be a morphism induced by the $M_G$-cover. Define $\bar{\phi}:G_{\mathbb{F}_{q}}\twoheadrightarrow N/G$ by $Frob_q\mapsto \tau^{e''}G$. Now, we have the following commutative diagram:

$$\xymatrix{
  1\ar[r] &  \Pi_1(\mathbb{P}^1_{\bar{\mathbb{F}}_{q}}\backslash B,b)\ar[r]\ar@{>>}[d]_{[\phi']}\ar@{>>}[ddr]^{\phi} & \Pi_1(\mathbb{P}^1_{\mathbb{F}_{q}}\backslash B,b)\ar[r]\ar@{>>}[d]_{[\tilde{\phi}']}\ar@{-->>}[ddr]^{\tilde{\phi}} & G_{\mathbb{F}_{q}}\ar[r]\ar@{>>}[d]_{[\bar{\phi}']}\ar@{>>}[ddr]^{\bar{\phi}} & 1 &
\\
 1\ar[r]  & G/Cen(G)\ar[r] & N/Cen_N(G)\ar[r] & N/GCen_N(G)\ar[r] & 1 &
\\
 & 1\ar[r] & G\ar[r]\ar@{>>}[ul] & N\ar[r]\ar@{>>}[ul] & N/G\ar[r]\ar@{>>}[ul] & 1
}$$

\noindent where $[\tilde{\phi}']$ is the map induced by $\tilde{\phi}'$ and the canonical isomorphism $M/Cen_{M}(G)\cong N/Cen_N(G)$. The obstruction to the existence of $\tilde{\phi}$ lies in $H^2(\mathbb{F}_q,Cen_N(G))$ \cite[Theorem 4.3]{de1}. Since $\mathbb{F}_q$ has cohomological dimension $1$, there exists such a lift $\tilde{\phi}:\Pi_1(\mathbb{P}^1_{\mathbb{F}_{q}}\backslash B,b)\rightarrow N$. Since $(e'',d)=1$, $\bar{\phi}$ is surjective and so is $\tilde{\phi}$. Hence, the diagram above corresponds to an $N_G$-cover, say $(f,\mathcal{S})$. By the construction, $(f,\mathcal{S})\in \Delta^{-1}(P)$ and this completes the proof.
\end{proof}

  We call a tuple $\bar{C}$ of conjugacy classes of $G$ \emph{Nielsen tuple} and its $N$-orbit $[\bar{C}]_N$ $N-Nielsen$ $tuple$. We will also refer to $N$-Nielsen tuples simply as Nielsen tuples when the meaning is clear from the context. Given two Nielsen tuples $\bar{C}$ and $\bar{C}'$, we write $\bar{C}=\bar{C}'$ if they differ only by a permutation of the entries. We denote their concatenation by $\bar{C}+\bar{C'}$.

\begin{dfn}
 Given an integer $1\leq e\leq d'$ with $(e,d')=1$ and Nielsen tuple $\bar{C}$, we say $\bar{C}$ is $\mathbb{F}_q-rational$ \emph{of type} $e$ if $\bar{C}^{q\tau^{-e}}=\bar{C}$.
\end{dfn}

 Our motivation for this definition is the following proposition.

\begin{prop}\label{prop_nil}
For every $e$ with $1\leq e \leq d'$ and $(e,d')=1$, we have
$$\mathcal{H}^{\zeta_e}_G(\mathbb{F}_q)=\bigsqcup_{\substack{\bar{C} \\ \bar{C}^{q\tau^{-e}}=\bar{C}}}\mathcal{H}^{\zeta_e}_{\bar{C}}(\mathbb{F}_q).$$
\noindent where the union runs over $\mathbb{F}_q$-rational Nielsen tuples of type $e$.
\end{prop}

\begin{proof}
 Let $\bar{C}$ be a Nielsen tuple. $\bar{C}=[\bar{C}]_G$ and $\mathcal{H}_{\bar{C}}\cong\mathcal{H}_{[\bar{C}]_G}$ under the isomorphism $\mathcal{H}_G\cong\mathcal{H}_{G,G}$. By Theorem \ref{thm_decomp} , $\bar{C}^{q\tau^{-e}}=\bar{C}$ if and only if the corresponding component $\mathcal{H}^{\zeta_e}_{\bar{C}}$ is defined over $\mathbb{F}_q$. Applying Theorem \ref{thm_decomp} for $\mathcal{H}_G$, we  get the desired decomposition.
\end{proof}

%-------------------------------------------------------------------------------------------------------------------

\section{Connected Components of Hurwitz Schemes}

 Given an $N_G$-cover $(f,\mathcal{S})$ defined over $\mathbb{F}_q$, by definition we have a surjective morphism $\tilde{\phi}:\Pi_1(\mathbb{P}^1_{\mathbb{F}_q}\backslash B,b)\rightarrow \rho(N)$ which is unique up to composition with an inner automorphism of $N$ \cite[Remark 3.4.1]{we}. If $L/\mathbb{F}_q(t)$ is the $N$-extension corresponding to $ker(\tilde{\phi})$, then $L=\mathbb{F}_q(Y)$ for some connected (not necessarily geometrically connected) curve $Y/\mathbb{F}_q$ and the extension $\mathbb{F}_q(Y)/\mathbb{F}_q(t)$ corresponds to a Galois $N-$cover $g:Y\rightarrow\mathbb{P}^1_{\mathbb{F}_q}$. Now, let $g':Y'\rightarrow \mathbb{P}^1_{\mathbb{F}_q}$ be the degree-$n$ cover associated to one-point stabilizer $H$.

\begin{dfn}
 The {\it discriminant} of the $N_G$-cover $(f,\mathcal{S})$ is the number $q^{r(f,\mathcal{S})}$ where $r(f,\mathcal{S})$ is the degree of the ramification divisor of the degree-$n$ cover $g':Y'\rightarrow \mathbb{P}^1_{\mathbb{F}_q}$.
\end{dfn}

 Define our counting function $\mathcal{Z}_{G,N}(\mathbb{F}_q,X)$ to be the number of isomorphism classes of $N_G$-covers $(f,\mathcal{S})$ defined over ${\mathbb{F}_q}$ with $q^{r(f,\mathcal{S})}<X$.

For a given $k$-tuple of conjugacy classes
$\bar{C}=(C_1,...,C_k)$, set $|\bar{C}|:=k$ and
$r(\bar{C}):=\sum_{i=1}^k ind(C_i)$ where $ind(C)$ is the index of
a representative of the conjugacy class $C$. Let $\Sigma_{r,e}$ be
the set of Nielsen tuples $\bar{C}$ of type $e$ with
$r(\bar{C})=r$; and let $\Sigma_r$ be the set of Nielsen tuples
$\bar{C}$ with $r(\bar{C})=r$. Let $\Sigma^N_r$ be the set of
$G-N$-Nielsen tuples $[\bar{C}]_N$ with $r(\bar{C})=r$. Finally,
let $\mathcal{H}_{{[\bar{C}]_N}}(\mathbb{F}_q)^N$ denote the
subset of $\mathcal{H}_{{[\bar{C}]_N}}(\mathbb{F}_q)$ consisting
of all the points in $\mathcal{H}_{{[\bar{C}]_N}}(\mathbb{F}_q)$
parameterizing $N_G$-covers.

 We will need the following counting functions:
\begin{dfn}
\begin{enumerate}
\item[i.]$h(q,r):=\sum_{[\bar{C}]_N\in \Sigma^N_r}|\mathcal{H}_{{[\bar{C}]_N}}(\mathbb{F}_q)^N|$.

\item[ii.]$h_1(q,r,e):=\sum_{\bar{C}\in \Sigma_{r,e}} |\mathcal{H}^{\zeta_e}_{\bar{C}}(\mathbb{F}_q)|$ and $h_1(q,r):=\sum_{\substack{1\leq e\leq d'\\ (e,d')=1}} h_1(q,r,e)$.

\item[iii.]$h_2(q,r,e):=\sum_{\bar{C}} q^{|\bar{C}|}$ where the sum runs over all geometrically connected components of $\mathcal{H}^{\zeta_e}_{\bar{C}}$'s defined over $\mathbb{F}_q$ where $\bar{C}\in \Sigma_{r,e}$ and $h_2(q,r):=\sum_{\substack{1\leq e\leq d'\\ (e,d')=1}} h_2(q,r,e)$.

\item[iv.]$h_3(q,r,e):=\sum_{\bar{C}\in \Sigma_{r,e}} q^{|\bar{C}|}$  and  $h_3(q,r):=\sum_{\substack{1\leq e\leq d'\\ (e,d')=1}} h_3(q,r,e)$.
\end{enumerate}
\end{dfn}

 Recall that we want to count the $N$-covers defined over $\mathbb{F}_q$. More precisely, we want to compute $\mathcal{Z}_{G,N}(\mathbb{F}_q,X)$ which is asymptotic to $\sum_{q^r<X}h(q,r)$ by Theorem \ref{thm_decomp}. On the other hand, by Proposition \ref{prop_dec}, we have $\sum_{q^r<X}h(q,r)\asymp \sum_{q^r<X}h_1(q,r)$. Thus, we get:

\begin{lemma}\label{lemma_h1} We have
$$\mathcal{Z}_{G,N}(\mathbb{F}_q,X)\asymp \sum_{q^r<X}h_1(q,r).$$
\end{lemma}

 Observe that $\sum_{q^r<X}h_2(q,r)$ is a good approximation to the desired sum $\sum_{q^r<X}h_1(q,r)$ on heuristic grounds. Therefore, in this section, our aim is to compute the sum $\sum_{q^r<X}h_2(q,r)$ (and then we will use the heuristic); this is not easy.

 If $\mathcal{H}^{\zeta_e}_{\bar{C}}$ were geometrically connected for all $\mathbb{F}_q$-rational of type $e$ Nielsen tuples $\bar{C}$ and for all relevant $e$, then the sum $\sum_{q^r<X}h_2(q,r)$ would be equal to $\sum_{q^r<X}h_3(q,r)$ and so we would reduce the problem to computing $\sum_{q^r<X}h_3(q,r)$.   Unfortunately, in general, this is not the case.

 Note that computing the sum $\sum_{q^r<X}h_3(q,r)$ boils down to computing the connected components of Hurwitz spaces and this is a very old combinatorial problem with a rich history, going all the way back to Hurwitz \cite{hr} and Clebsch \cite{cl}.

 In this section, we will see that there are "many" $\mathbb{F}_q$-rational Nielsen tuples $\bar{C}$ such that  $\mathcal{H}^{\zeta_e}_{\bar{C}}$ possesses a geometrically connected component defined over $\mathbb{F}_q$ for all $e$ and the number of these components is bounded by positive constant depending only on the group $N$. Thus, we will reduce the problem to computing $\sum_{q^r<X}h_3(q,r)$. More precisely, the purpose of this section is to prove the following proposition.

\begin{prop}\label{prop_main}
Let $e$ be such that $1\leq e\leq d'$ and $(e,d')=1$. There exist positive constants $m, c_1$ depending on $N$ such that
$$\sum_{r<R-m}h_3(q,r,e)<\sum_{r<R}h_2(q,r,e) < c_1 \sum_{r<R}h_3(q,r,e).$$
\end{prop}

 By the lemma below, we have the right-hand-side inequality.

\begin{lemma}\label{lemma_upb}\cite[Lemma 3.3]{el2}
There exists a constant $c_1$ such that $n(\bar{C})< c_1$ for all $\bar{C}$ where $n(\bar{C})$ is the number of $\Pi_1(U_k)$-orbits in $Ni({\bar{C}})$.
\end{lemma}

 As for the inequality on the left-hand-side, we need a result controlling the geometrically connected components of Hurwitz spaces. The first such theorem along the lines presented here is attributed to Conway and Parker-- the first version of such a theorem to appear in print is due to Fried and V\"olklein \cite{fr}.

\begin{lemma}\label{lemma_cp}(Conway-Parker)
\cite[Appendix]{fr} Let $\tilde{G}'$ be a finite group such that the Schur multiplier $M(\tilde{G}')$ is generated by commutators. Then, there exists a constant $K$ such that for any Nielsen tuple $\bar{E}$ of $\tilde{G}'$ which contains at least $K$ copies of each nontrivial conjugacy classes of $\tilde{G}'$ the corresponding Hurwitz space $\mathcal{H}_{\bar{E}}$ is geometrically connected.
\end{lemma}

 We will need the next two technical lemmas to prove Proposition \ref{prop_main}. First of them is needed to apply Lemma \ref{lemma_cp} to our setting.

\begin{lemma}\label{lemma_diag}
There exists a commutative diagram of finite groups
$$\xymatrix{\tilde{N}\ar@{>>}[r]^\pi & N \\
\tilde{G}\ar@/_/[rr]_\pi\ar@{>>}[r]^{\pi'}\ar@{^{(}->}[u] & \tilde{G}'\ar@{>>}[r] & G\ar@{^{(}->}[ul] }$$
such that the Schur multiplier $M(\tilde{G}')$ is generated by commutators, where the vertical maps are inclusions and horizontal maps are surjections.
\end{lemma}

\begin{proof}
 By Lemma 1 of \cite{fr}, there exists an extension of groups $$\xymatrix{1\ar[r] & M'\ar[r] & \tilde{G}'\ar[r] & G\ar[r] & 1}$$ such that Schur multiplier $M(\tilde{G}')\cong M'$ is generated by commutators. Let $M:=Ind_N^G(M')$ be the induced $N$-module. Then, by Shapiro's lemma, $H^2(G,M')$ is isomorphic to $H^2(N,M)$. Let $$\xymatrix{1\ar[r] & M\ar[r] & \tilde{N}\ar[r]^\pi & N\ar[r] & 1}$$ be an extension which corresponds to the extension above via the isomorphism $H^2(G,M')\cong H^2(N,M)$, and let $\tilde{G}=\pi^{-1}(G)$. Now, the evaluation morphism $M\rightarrow M'$ defined by $f\mapsto f(1)$ induces the following surjective morphism of groups $\pi':\tilde{G}\twoheadrightarrow \tilde{G}'$. One can easily check that these extensions fits into the desired commutative diagram above and complete the proof.
\end{proof}

 Let $\pi:\tilde{N}\rightarrow N$ be a surjective morphism of finite groups with $\tilde{G}:=\pi^{-1}(G)$ as in Lemma \ref{lemma_diag} . Clearly, $\tilde{G}$ is normal subgroup of index $d$ with cyclic quotient. Let's fix and denote $\tilde{\tau}\in \tilde{N}$ generating $\tilde{T}:=\tilde{N}/\tilde{G}$ with $\pi(\tilde{\tau})=\tau$. Let $\tilde{T}':=\tilde{N}/\tilde{G}Cen_{\tilde{N}}\tilde{G}$. Simplifying the notation,  $\tilde{\tau}$ will also denote its image in $\tilde{T}$ and $\tilde{\tau}_1$ will denote its image in $\tilde{T}'$. Note that the map $\bar{\pi}:\tilde{T}'\rightarrow T'$ induced by $\pi$ is an isomorphism. Note also that $\tilde{N}$ acts on geometrically connected Galois $\tilde{G}$-covers and this induces an action of $\tilde{T}'$ on $\mathcal{H}_{\tilde{G}}$. Namely, if a geometrically connected Galois $G$-cover is induced by a morphism $\phi:\pi_1(\mathbb{P}^1_{\bar{\mathbb{F}}_q}\backslash B,b)\rightarrow \tilde{G}$ then  the action of $x\in N$ is defined by $\phi^x(\gamma):=x^{-1}\phi(\gamma)x$ for all $\gamma\in \pi_1(\mathbb{P}^1_{\bar{\mathbb{F}}_q}\backslash B,b)$. One defines the cocycles $\tilde{\zeta_e}\in H^1(\mathbb{F}_q, \tilde{T}')$ by $Frob_q\mapsto  \tilde{\tau}_1^e$ for all $e$ with $1\leq e \leq d'$.

 %Clearly, the pushforward $\bar{\pi}_*:H^1(\mathbb{F}_q, \tilde{T}')\tilde{\rightarrow} H^1(\mathbb{F}_q, T')$ takes $\tilde{\zeta}_e$ to $\zeta_e$ for all $e$ with $1\leq e \leq d$.

 Here is the second lemma we need to prove Proposition \ref{prop_main}:

\begin{lemma}\label{lemma_main}
For every $e$ with $1\leq e\leq d'$ and $(e,d')=1$ there exists a finite set of $\mathbb{F}_q$-rational of type $e$ Nielsen tuples $\bar{D}_1,...,\bar{D}_r$  in $\tilde{G}$ such that for any given $\mathbb{F}_q$-rational of type $e$ Nielsen tuple $\bar{D}$, there exists $D_i$ which makes $\mathcal{H}_{\bar{D}+\bar{D}_i}(\bar{\mathbb{F}}_q)$ nonempty.
\end{lemma}

\begin{proof}

 We can replace $\tilde{\tau}$ with $\tilde{\tau}^e$ in the proof below so we may assume that $e=1$. Let $H\leq \tilde{G}$ be the subgroup generated by $g_1...g_k$ such that $\bar{D}^q=\bar{D}^{\tilde{\tau}}$ where $\bar{D}$ is the Nielsen tuple of $(g_1,...,g_k)$. We will show that for all $h\in H$, there exists a $k$-tuple $\bar{g}=(g_1,..,g_k)\in \tilde{G}^k$ such that

 -$g_1...g_k=h$

 -$\tilde{G}=<g_1,..,g_k>$

 -$\bar{g}$ represents a $\mathbb{F}_q$-rational Nielsen tuple $\bar{D}$ in $\tilde{G}$ with $\bar{D}^q=\bar{D}^{\tilde{\tau}}$.

 It suffices to show for $h=1$: If $\bar{g}=(g_1,...,g_k)$ is such a tuple with $g_1...g_k=1$ and if $\bar{h}=(h_1,...,h_k)$ is a $k$-tuple having product $h$ and representing a Nielsen tuple $\bar{D'}$ with $\bar{D'}^q=\bar{D'}^{\tilde{\tau}}$, then we can just concatenate $\bar{g}$ with $\bar{h}$ and thus, get such a tuple having multiple $h$.

 So, let $\bar{g}=(h_1,...,h_s)$ be a generating set for $\tilde{G}$ which represents a Nielsen tuple $\bar{D}_0$. Let $\bar{D}$ be the concatenation of the tuples in $\tilde{T}_1$-orbit of $\bar{D}_0$. Obviously, $\bar{D}^q=\bar{D}^{\tilde{\tau}}$. Now, if $(g_1,...,g_k)\in \bar{D}$ with $x=g_1...g_k$ then take $|x|$-multiple of $(g_1,...,g_k)$ where $|x|$ denotes the order of $x$. This completes the proof of the claim.

\end{proof}

 Let $1\leq e\leq d'$ with $(e,d')=1$ and $t_e=q\tau^{-e}$. $t_e$ acts on the set of conjugacy classes of $G$ and $t_e^r$ acts trivially for some positive integer $r$ since $(q,|G|)=1$. Therefore, one can put a group structure on the set $T_e:=\{t_e^i| i=1,...,r\}$. Given a conjugacy class $\mathcal{O}$ in $G$, we will denote the $T_e$-orbit of $\mathcal{O}$ by $\mathcal{C}_e(\mathcal{O})$ and define
    $$\bar{\mathcal{O}}_e:=\sum_{\bar{C}\in \mathcal{C}_e(\mathcal{O})}\bar{C}.$$
\noindent Notice that, given a Nielsen tuple $\bar{C}$,  the condition $\bar{C}^q=\bar{C}^{\tau^e}$ is equivalent to the condition $\bar{C}^{q\tau^{-e}}=\bar{C}$ so $\bar{\mathcal{O}}_{e}$ is the "smallest" $\mathbb{F}_q$-rational Nielsen tuple of type $e$ containing $\mathcal{O}$. Thus, given a Nielsen tuple $\bar{C}$ in $G$ with $\bar{C}^q=\bar{C}^{\tau^e}$, we can write
$$\bar{C}=\sum_{\bar{\mathcal{O}}_e}a_{\mathcal{O}}\bar{\mathcal{O}}_e.$$

 Likewise, we can consider the set $\tilde{T}_e=\{\tilde{t}^i_e|i=1,...,s\}$ where $\tilde{t}_e=q\tilde{\tau}^{-e}$ and the Nielsen tuple $\bar{\tilde{\mathcal{O}}}_e$ where $\tilde{\mathcal{O}}$ is a conjugacy class of $\tilde{G}$ which projects down to $\mathcal{O}$.

 We follow \cite{el2} closely to prove Proposition \ref{prop_main}:
 \ \\

\noindent {\it Proof} of Proposition \ref{prop_main}. We want to prove the inequality on the left-hand-side.

 Let $e$ be such that $1\leq e\leq d'$ and $(e,d')=1$. Let $\tilde{N}$, $\tilde{G}$ and $\tilde{G}'$ be as in Lemma \ref{lemma_diag} and let $\bar{D}_1,...,\bar{D}_r$ be the Nielsen tuples in $\tilde{G}$ as in Lemma \ref{lemma_main}. For each conjugacy class $\mathcal{O}$ in $G$ fix  a conjugacy class $\tilde{\mathcal{O}}$ in $\tilde{G}$ with  $\pi(\tilde{\mathcal{O}})=\mathcal{O}$. Note that $\pi(\bar{\tilde{\mathcal{O}}}_e)=b_{\mathcal{O}}\bar{\mathcal{O}}_e$ for some $b_{\mathcal{O}}$. Given a Nielsen tuple $\bar{C}$ in $G$ with $\bar{C}^q=\bar{C}^{\tau^e}$, we can write $\bar{C}=\sum_{\bar{\mathcal{O}}_e}a_{\mathcal{O}}\bar{\mathcal{O}}_e$ and the following Nielsen tuple in $\tilde{G}$
$$\bar{D}:=\sum_{\bar{\mathcal{O}}_e} \lceil \frac{a_{\mathcal{O}}}{b_{\mathcal{O}}} \rceil \bar{\tilde{\mathcal{O}}}_e$$
\noindent is $\mathbb{F}_q$-rational of type $e$ and $\pi(\bar{D})=\bar{C}+\bar{C}'$ where $\bar{C}'$ can be drawn from a finite set of Nielsen tuples $\bar{C}_1',...,\bar{C}_s'$ with $\bar{C}_i'^q=\bar{C}_i'^{\tau^e}$ for all $i$.

 Now, fix a Nielsen tuple $\bar{B}$ in $\tilde{G}$ with $\bar{B}^q=\bar{B}^{\tilde{\tau}^e}$ such that its projection $\pi'(\bar{B})$ in $\tilde{G}'$ contains at least $K$ nontrivial conjugacy classes of $\tilde{G}$ where $K$ is the constant in Lemma \ref{lemma_cp}. By Lemma \ref{lemma_main}, there exists an $i$ such that $\bar{D}+ \bar{B}+\bar{D}_i$ is $\mathbb{F}_{q}$-rational of type $e$ and $\mathcal{H}^{\tilde{\zeta}_e}_{\bar{D}+ \bar{B}+\bar{D}_i}(\bar{\mathbb{F}}_q)$ is nonempty and, by Theorem \ref{thm_decomp}, $\mathcal{H}^{\tilde{\zeta}_e}_{\bar{D}+ \bar{B}+\bar{D}_i}$ defined over $\mathbb{F}_q$.

 The projection $\pi(\bar{D}+ \bar{B}+\bar{D}_i)$ is $\mathbb{F}_q$-rational of type $e$ and it can be expressed as $\bar{C}+\bar{C}_j+n.\bar{1}$ where $\bar{C}_j$ can be drawn from a finite set of Nielsen tuples $\bar{C}_1,...,\bar{C}_k$ and $\bar{1}$ denotes the trivial conjugacy class.

 We now claim that $\mathcal{H}^{\zeta_e}_{\bar{C}+\bar{C}_j}$ has an $\mathbb{F}_q$-rational geometrically connected component. Once again note that $\mathcal{H}_{G,G}\cong \mathcal{H}_G/\mathbb{F}_q$ where $\mathcal{H}_G$ is the coarse moduli scheme of geometrically connected Galois $G$-covers. So, for any Galois $\tilde{G}$-cover $Y\rightarrow \mathbb{P}^1$ we have the canonically associated Galois $G$-cover  $Y/U\rightarrow \mathbb{P}^1$ where $U:=\text{ker}(\pi:\tilde{G}\twoheadrightarrow G)$. This defines a morphism of schemes$/\mathbb{F}_q$ $$\pi_*:\mathcal{H}^{\tilde{\zeta}_e}_{\bar{D}+ \bar{B}+\bar{D}_i} \rightarrow \mathcal{H}^{\zeta_e}_{\bar{C}+\bar{C}_j}.$$ Notice that this map factors through (possibly over $\bar{\mathbb{F}}_q$) the natural map, which is induced likewise, $$\pi'_*:\mathcal{H}^{\tilde{\zeta}_e}_{\bar{D}+ \bar{B}+\bar{D}_i} \rightarrow \mathcal{H}_{\pi'(\bar{D}+ \bar{B}+\bar{D}_i)}.$$ Since $\pi'(\bar{B})$ contains at least $K$ nontrivial conjugacy classes of $\tilde{G}'$, by Lemma \ref{lemma_cp}, the image of $\pi'_*$ is geometrically connected and so is the image of $\pi_*$. The image of $\pi_*$ is the $\mathbb{F}_{q}$-rational geometrically connected component we want.

 Define $h_2(q,\bar{C})$ to be the number of $\mathbb{F}_q$-rational geometrically connected components of $\mathcal{H}^{\zeta_e}_{\bar{C}}$ multiplied by $q^{|\bar{C}|}$. By the discussion above, $h_2(q,\bar{C}+\bar{C}_j)\geq q^{|\bar{C}+\bar{C}_j|}$ for some $j$. For each $\mathbb{F}_q$-rational of type $e$ Nielsen tuple $\bar{C}$, fix such a $\bar{C}_j$ and set $\bar{C}^{pr}:=\bar{C}+\bar{C}_j$.

 Thus, we have

 $$\sum_{\substack{\bar{C}\\ r(\bar{C})<R}} h_2(q,\bar{C}^{pr})\geq \sum_{\substack{\bar{C}\\ r(\bar{C})<R}} q^{|\bar{C}^{pr}|} > \sum_{r<R} h_3(q,r,e) $$

\noindent and

$$\sum_{\substack{\bar{C}\\r(\bar{C})<R}} h_2(q,\bar{C}^{pr}) \leq \sum_{\substack{\bar{C}\\r(\bar{C})<R+m}} h_2(q,\bar{C})= \sum_{r<R+m} h_2(q,r,e)$$

\noindent where $m$ is the supremum of $r(\bar{C}_j)$'s and the sums run over $\mathbb{F}_q$-rational of type $e$ Nielsen tuples $\bar{C}$. This completes the proof.

%----------------------------------------------------------------------------------------------------------------------

\section{Malle's Conjecture}

 Let $\mathcal{Z}'_{G,N}(\mathbb{F}_q,X)=\sum_{q^r<X}h_2(q,r)$. Observe that heuristic $(A)$ implies $h_1(q,r)=h_2(q,r)$ (this indeed is the only point where we use the heuristic). So, by lemma \ref{lemma_h1}, $\mathcal{Z}_{G,N}(\mathbb{F}_q,X)$  has the same asymptotic order with the counting function $\mathcal{Z}'_{G,N}(\mathbb{F}_q,X)$ on heuristic grounds. More precisely, we have

\begin{lemma}\label{lemma_change}
 Assume heuristic $(A)$. If $G$ splits in $N$, then we have $$\mathcal{Z}_{G,N}(\mathbb{F}_q,X)\asymp \mathcal{Z}'_{G,N}(\mathbb{F}_q,X).$$
\end{lemma}

 In this section, our main purpose is to compute asymptotic order of the counting function $\mathcal{Z}'_{G,N}(\mathbb{F}_q,X)$. In detail: we will define the constant $b(G,N,\mathbb{F}_q)$ and prove Theorem \ref{main_thm}. We will write a conjecture for the counting function $\mathcal{Z}_N(k,X)$ for any global field $k$ and, using Theorem \ref{main_thm}, we will show that our conjecture gives the right asymptotic in some important cases.

 We will need the following lemma from Tauberian theory, for a proof see \cite[Lemma 2.3]{el2}:

\begin{lemma}\label{lemma_an}
 Suppose $\{a_n\}$ is a sequence of real numbers with $a_n=0$ whenever $n$ is not a power of $q$, and suppose
$$\sum_{r=1}^\infty a_{q^r} q^{-rs}$$
considered as a formal power series, is a rational function $f(t)$ of $t=q^s$. Let a be a positive real number. If $f(t)$ has no poles with $|t|\geq q^a$, then
$$\sum_{n=1}^{X}a_n\ll X^a.$$
If $f(t)$ has a pole of order $b$ at $t=q^a$ and no other poles with $|t|\geq q^a$, then
$$\sum_{n=1}^X a_n\asymp  X^a(log X)^{b-1}.$$
\end{lemma}

 Let $1\leq e\leq d'$ with $(e,d')=1$ and $t_e=q\tau^{-e}$. Note that $t_e$ acts on the set of conjugacy classes of $G$ and that $t_e^r$ acts trivially for some positive integer $r$. Set $T_e:=\{t_e^i| i=1,...,r\}$. Given a conjugacy class $\mathcal{O}$ in $G$, we denote the $T_e$-orbit of $\mathcal{O}$ by $\mathcal{C}_e(\mathcal{O})$ and define $\bar{\mathcal{O}}_e:=\sum_{\bar{C}\in \mathcal{C}_e(\mathcal{O})}\bar{C}$.

 Let $d'=[N:GCen_N(G)]$ and $d''=[GCen_N(G):G]$ \emph{i.e.} $d=d'd''$. Denote the set of $G-$conjugacy classes of minimal-index elements of $G$ by $\mathcal{C}(G)$. Given $e$ with $1\leq e\leq d'$ and $(e,d')=1$, define
$$\mathcal{C}_e(G,N,\mathbb{F}_q):=\{ \bar{\mathcal{O}}_e: \mathcal{O}\in \mathcal{C}(G)\} \text{\indent and \indent} b_e(G,N,\mathbb{F}_q):=|\mathcal{C}_e(G,N,\mathbb{F}_q)|.$$
\noindent Finally, set
$$b(G,N,\mathbb{F}_q)=max\{b_e(G,N,\mathbb{F}_q)|1\leq e\leq d'\text{ and } (e,d')=1\}.$$

 By Lemma \ref{lemma_change}, the following theorem (respectively, the related corollaries) is just another way to state Theorem \ref{main_thm} (respectively, the corollaries) in the introduction:

\begin{thm}\label{thm_main}
If $G$ splits in $N$, then we have
$$\mathcal{Z}'_{G,N}(\mathbb{F}_q,X)\asymp X^{a(G)} (logX)^{b(G,N,\mathbb{F}_q)-1}.$$

\end{thm}

\begin{proof}

By Proposition \ref{prop_main}, we have

$$\sum_{q^r<X} h_2(q,r)\asymp \sum_{q^r<X} h_3(q,r)=\sum_{\substack{1\leq e\leq d' \\ (e,d')=1}}\sum_{q^r<X} h_3(q,r,e).$$

\noindent On the other hand, for every $e$ with $1\leq e\leq d'$ and $(e,d')=1$, we have the factorization:

$$\sum_{r=1}^{\infty} h_3(q,r,e)q^{-rs}=\sum_{\substack{\bar{C} \\ \bar{C}^q=\bar{C}^{\tau^e}}} q^{|\bar{C}|q^{-r(\bar{C})s}} = \prod_{\bar{\mathcal{O}}_e} \frac{1}{1-q^{|\bar{\mathcal{O}}_e|(1-ind(\bar{\mathcal{O}}_e)s)}} $$

\noindent where the product is indexed by $T_e$-orbits $\bar{\mathcal{O}}_e$ of conjugacy classes $\mathcal{O}$ of $G$, $ind(\bar{\mathcal{O}}_e)$ is the index of an element of $\mathcal{O}$ and $|\bar{\mathcal{O}}_e|$ denotes the number of conjugacy classes in the orbit $\bar{\mathcal{O}}_e$. By Lemma \ref{lemma_an},

$$\sum_{q^r<X} h_3(q,r,e) \asymp X^{a(G)} (logX)^{b_e(G,N,\mathbb{F}_q)-1}$$
\noindent and so
$$\sum_{\substack{1\leq e\leq d' \\ (e,d')=1}}\sum_{q^r<X} h_3(q,r,e) \asymp X^{a(G)} (logX)^{b(G,N,\mathbb{F}_q)-1}$$

\noindent for sufficiently large $X$. Putting them all together, we get
$$\sum_{q^r<X} h_2(q,r)\asymp  X^{a(G)} (logX)^{b(G,N,\mathbb{F}_q)-1}$$

\noindent and this completes the proof.

\end{proof}

As a special case, we get the result of Ellenberg-Venkatesh \cite{el2}:

\begin{cor}\label{cor_ell}
$\mathcal{Z}'_{N,N}(\mathbb{F}_q,X)\asymp X^{a(N)}(logX)^{b(N,\mathbb{F}_q)-1}.$
\end{cor}

\begin{proof}
We just need to show that $b(N,N,\mathbb{F}_q)=b(N,\mathbb{F}_q)$.  Using the notation above, we have $d=e=1$ and $\tau=1$ in $N$. So, $\mathcal{C}_e(N,N,\mathbb{F}_q)=\mathcal{C}(N)/G_{\mathbb{F}_q}$ and $b(N,N,\mathbb{F}_q)=b_e(G,N,\mathbb{F}_q)=|\mathcal{C}(N)/G_{\mathbb{F}_q}|=b(N,\mathbb{F}_q)$.
\end{proof}

 Let $N=(\langle(123)\rangle\oplus \langle(456)\rangle)\rtimes \langle(14)(25)(36)\rangle\leq S_6$ and let $\mathcal{Z}_N(\mathbb{Q},\mathbb{Q}(\zeta_3), X)$ be the number of isomorphism classes of $N$-extensions $K/\mathbb{Q}$ containing $\mathbb{Q}(\zeta_3)$ such that $d_{K/\mathbb{Q}}<X$. In \cite{kl2}, Kl\"uners shows that $\mathcal{Z}_N(\mathbb{Q}, X)\asymp X^{1/2}logX$ contradicting with Malle's conjecture which predicts $\mathcal{Z}_N(\mathbb{Q},X)\asymp X^{1/2}$. Indeed, he proves that $\mathcal{Z}_N(\mathbb{Q},\mathbb{Q}(\zeta_3), X)\asymp X^{1/2}logX$. With evident modifications, one gets the same result in function field case for $q=2$ $(mod3)$ \emph{i.e.} $\mathcal{Z}_N(\mathbb{F}_q,\mathbb{F}_{q^2},X)\asymp X^{1/2}logX$. The following corollary shows that our theorem gives the right asymptotic for the number of $N$-extensions in this case.

\begin{cor}\label{cor_coun}
$\mathcal{Z}'_N(\mathbb{F}_q,X)\asymp X^{1/2}logX$ where  $N=(\langle(123)\rangle\oplus \langle(456)\rangle)\rtimes \langle(14)(25)(36)\rangle$ and $q=2$ $(mod3)$.
\end{cor}

\begin{proof}
Given an $N$-extension $K/\mathbb{F}_q(t)$, maximal constant
intermediate subfield in its Galois closure
$\widehat{K/\mathbb{F}_q(t)}$ might be $\mathbb{F}_q(t)$,
$\mathbb{F}_{q^2}(t)$ or $\mathbb{F}_{q^6}(t)$ corresponding to
the normal subgroups of $N$, respectively, $N$,
$G_1:=\langle(123)\rangle\oplus \langle(456)\rangle$ or
$G_2:=\langle(123)(456)\rangle$. Therefore, we have
$$\mathcal{Z}'_N(\mathbb{F}_q(t),X)\asymp \mathcal{Z}'_{N,N}(\mathbb{F}_q,X)+\mathcal{Z}'_{G_1,N}(\mathbb{F}_q,X)+\mathcal{Z}'_{G_2,N}(\mathbb{F}_q,X).$$

 Using Theorem \ref{thm_main}, one can easily see that
$$\mathcal{Z}'_{N,N}(\mathbb{F}_q,X)\asymp X^{1/2} \text{\indent and \indent} \mathcal{Z}'_{G_2,N}(\mathbb{F}_q,X)\asymp X^{1/4}.$$\noindent We will show that $\mathcal{Z}'_{G_1,N}(\mathbb{F}_q,X)\asymp X^{1/2}logX$. Using the notation above, we have: $$\mathcal{C}(G)=\{(123),(132),(456),(465)\},$$ $a(G)=1/2$, $d=d'=2$, $e=1$ and $\tau=(14)(25)(36)$. We have two $T_e$-orbits: $\{(123),(465)\}$ and $\{(132),(456)\}$. Therefore, $b(G_1,N,\mathbb{F}_q)=b_e(G_1,N,\mathbb{F}_q)=2$. Hence, we are done.
\end{proof}

 Another interesting example in $\cite{kl2}$ is $N\cong (C_3\wr C_3)\times C_2\subseteq S_{18}$ where $C_3$ denotes the cyclic group of order $3$. Kl\"uners shows that $\mathcal{Z}_N(\mathbb{Q},X)<<X^{1/4}$. In function field case,  our main result gives us the exact asymptotic:

 \begin{cor}\label{cor_exam}
 $\mathcal{Z}'_N(\mathbb{F}_q(t),X)\asymp X^{1/4}$ where $q=2$ $(mod3)$ and $N\cong (C_3\wr C_3)\times C_2\subseteq S_{18}$.
 \end{cor}

 \begin{proof} Let's write $N=C_3\wr C_3 \times C_2= (\langle g_1\rangle \times \langle g_2\rangle \times \langle g_3\rangle)\rtimes \langle x\rangle \times \langle y\rangle$. We will just consider the normal subgroups $G$ with $a(G)=a(N)$. Given a $N$-extension $K/\mathbb{F}_q(t)$, $N$ has four such normal subgroups $G$ with cyclic quotient corresponding to possible maximal constant subextensions in the Galois closure of $K/\mathbb{F}_q(t)$. So, we have four cases.

{\it Case $1$}: $G=N$. In this case, $\mathcal{C}(G)=\{ \tilde{g}_1,\tilde{g}_1^2 \}$ where $\tilde{g}_i$ is the conjugacy class of $g_i$. So, $a(N)=1/4$. Since $q=2$ $(mod3)$ and $d=d'=e=1$, there is only one $T_e$-orbit, namely $\{ \tilde{g}_1,\tilde{g}_1^2 \}$. Thus, $b(N,N,\mathbb{F}_q)=1$ and $\mathcal{Z}'_{G,N}(\mathbb{F}_q,X)\asymp X^{1/4}.$

{\it Case $2$}: $G=C_3\wr C_3$. We have $[N:G]=2$, $a(G)=1/4$ and $d'=e=1$ since $y\in Cen_N(G)$. With the notation above, we have $\tau=y$ and $\mathcal{C}(G)=\{ \tilde{g}_1,\tilde{g}_1^2 \}$. Since $\tau\in Cen_N{G}$, the only $T_e$-orbit is $\{ \tilde{g}_1,\tilde{g}_1^2 \}$ and so $b(G,N,\mathbb{F}_q)=1$. Hence, $\mathcal{Z}'_{G,N}(\mathbb{F}_q,X)\asymp X^{1/4}$

{\it Case $3$}: $G=C_3\times C_3\times C_3$. We have $[N:G]=6$, $a(G)=1/4$, $d'=3$ and $\tau=xy$. We have
$\mathcal{C}(G)=\{ g_1,g_1^2,g_2,g_2^2,g_3,g_3^2 \}$ and $e=1$ or $e=2$. One can easily see
$$\mathcal{C}_1(g_i^{\epsilon})=\mathcal{C}_2(g^{\epsilon}_i)=\{g_1, g_1^2, g_2, g_2^2, g_3, g_3^2\}$$
\noindent for $i=1,2,3$ and $\epsilon=1,2.$ Therefore, there is only one $T_e$-orbit and $b(G,N,\mathbb{F}_q)=1$. Hence, $\mathcal{Z}'_{G,N}(\mathbb{F}_q,X)\asymp X^{1/4}.$

{\it Case $4$}: $G=C_3\times C_3\times C_3 \times C_2$. Using the same argument in the previous case, we see that $\mathcal{Z}'_{G,N}(\mathbb{F}_q,X)\asymp X^{1/4}$. QED

 \end{proof}

 The following corollary shows that our result coincides with Malle's conjecture for abelian groups, which is a result of Wright \cite{wr} in the case the base field is $\mathbb{Q}$. By Corollary \ref{cor_ell} and Lemma \ref{lemma_change}, we just need to show that extensions without a constant subextension is "more" than the extensions with a constant subextension (whose corresponding subgroup $G$ splits in $N$).

\begin{cor}\label{cor_abel}
Let $N$ be an abelian group. Then, $\mathcal{Z}'_{G,N}(\mathbb{F}_q,X)\ll \mathcal{Z}'_{N,N}(\mathbb{F}_q,X)$ for any normal subgroup $G$ with a cyclic complement.
\end{cor}

\begin{proof}
 Let $G$ be such a subgroup. Then, $a(G)\leq a(N)$. If $a(G) < a(N)$, then we are done. Assume $a(G)=a(N)$. Since $N$ is abelian, $d'=e=1$ and $\mathcal{C}_1(G,N,\mathbb{F}_q)=\mathcal{C}(G)/G_{\mathbb{F}_q}$. On the other hand, $\mathcal{C}(N,N,\mathbb{F}_q)=\mathcal{C}(N)/G_{\mathbb{F}_q}$ and $\mathcal{C}(G)\subseteq \mathcal{C}(N)$.  So, $b(G,N,\mathbb{F}_q)\leq b(N,N,\mathbb{F}_q)$ and we are done.
\end{proof}

 Note that one can revise Malle's conjecture (in other words, the constant $b(N,\mathbb{F}_q)$ in the conjecture) for function fields by just taking the maximum of all the constants $b(G,N,\mathbb{F}_q)$ over all normal subgroups $G\leq N$ with cyclic quotient and $a(G)=a(N)$. Notice that the constant $b(G,N,\mathbb{F}_q)$ is defined for any such subgroup $G$ (which is not necessarily split in $N$). 

 Now, we want to conclude our paper with the statement of the revised conjecture in number field case. First, we will introduce some notation. Let $k$ be a number field. We need to consider normal subgroups $G\leq N$ with abelian quotient $N/G$. Let $k^c$ be the maximal cyclotomic extension of $k$ (in a fixed algebraic closure $\bar{k}$). Set $C:=Cen_N(G)$. Given a cocycle $\varphi \in {\rm Hom}(G(k^c/k),N/G)$ and a conjugacy class $\bar{g}$ in $G$, define $\varphi${\it -twisted action} of $\sigma\in G_k$ on $\bar{g}$ by $$\sigma(\bar{g}):= \bar{g}^{\chi(\sigma)\varphi({\rm Res}(\sigma))^{-1}}$$ where ${\rm Res}:G_k\rightarrow G(k^c/k)$ is the restriction, $\chi$ is the cyclotomic character and $\varphi({\rm Res}(\sigma))^{-1}$ acts on the conjugacy class by conjugation. By analogy, we define $$b_{\varphi}(G,N,k):=|\mathcal{C}(G)/G_k|$$ where the $G_k$-action in question is the $\varphi$-twisted action. Finally, set $$b(G,N,k)=max\{b_{\varphi}(G,N,k): \varphi\in {\rm Hom}(G(k^c/k),N/G) \text{ and } \varphi \text{ is surjective} \}.$$ Based on our result, we propose the following correction to Malle's conjecture for any finite transitive subgroup $N\subseteq S_n$:
\ \\

\begin{con}
Fixing $N$ and $k$, set $$b(N,k)=max\{ b(G,N,k): N/G \text{ is abelian, and } a(G)=a(N)\}.$$ Then, we have $$\mathcal{Z}_N(k,X)\asymp X^{a(N)} (logX)^{b(N,k)-1}.$$
\end{con}
\ \\

 We remark that there are counterexamples in \cite{kl2} which are different than the one in Corollary \ref{cor_coun}, but in the same spirit. Using Theorem \ref{thm_main}, one can show that our result gives the right asymptotic in these cases.

 Summing it up, in the case of function fields, corollaries of our result shows that the conjecture above coincides with the existing results on Malle's conjecture and eliminates all the counterexamples known (to our knowledge) so far. Hence, it provides strong evidence in the favor of the conjecture above.

\ \\

\ \\

   \vspace{.2 in}
   \textsc{department of mathematics, university of wisconsin, 480 Lincoln Dr Madison wi 53706}

   \textit{E-mail address:} turkelli@math.wisc.edu

\end{document}